\documentclass[11pt,reqno]{amsart}
\usepackage{amssymb,latexsym,mathrsfs}
\usepackage{url}
\usepackage{fullpage}
\usepackage{mathabx}
\usepackage{enumerate}[1.)]
\numberwithin{equation}{section}
\usepackage{hyperref}

\renewcommand{\eqref}[1]{(\ref{#1})}   
\theoremstyle{plain}
\newtheorem{theorem}{Theorem}[section]

\newtheorem{corollary}[theorem]{Corollary}

\newtheorem{lemma}[theorem]{Lemma}
\newtheorem{proposition}[theorem]{Proposition}

\theoremstyle{definition}

\theoremstyle{remark} 

\newtheorem{remark}[theorem]{Remark}

\newcommand{\Fp}{\mathbb{F}_p}
\newcommand{\Fpt}{\mathbb{F}^\times_p}
\newcommand{\Cc}{\mathbb{C}}
\newcommand{\Rr}{\mathbb{R}}
\newcommand{\Zz}{\mathbb{Z}}

\newcommand{\LL}{\mathcal{L}}
\newcommand{\mcL}{\mathcal{L}}

\newcommand{\tA}{\mathsf{A}} 
\newcommand{\tB}{\mathsf{B}}
\newcommand{\tC}{\mathsf{C}} 
\newcommand{\tD}{\mathsf{D}}

\DeclareMathOperator{\cond}{cond}
\newcommand{\sheaf}[1]{\mathcal{{#1}}}

\newcommand{\mods}[1]{\,(\mathrm{mod}\,{#1})}
\newcommand{\ov}[1]{\overline{#1}}

\newcommand{\lra}{\longrightarrow}

\newcommand{\eps}{\varepsilon}
\DeclareMathOperator{\hypk}{\mathrm{Kl}}

\newcommand{\bessel}[1]{\widecheck{{#1}}}
\newcommand{\bbessel}[1]{\widetilde{{#1}}}
\newcommand{\uple}[1]{\text{\boldmath${#1}$}}

\newcommand\sumsum{\mathop{\sum\sum}\limits}
\newcommand\sumtriple{\mathop{\sum\ \sum\ \sum}\limits}
\def\stacksum#1#2{{\stackrel{{\scriptstyle #1}}
{{\scriptstyle #2}}}}

\def\wihat#1{\widehat{#1}}

\newcommand{\fourier}[1]{\widehat{{#1}}}

\begin{document}

\title[Exponent of distribution of $d_3$]{On the exponent of
  distribution of the ternary divisor function}

\date{\today}

\author[\'E. Fouvry]{\'Etienne Fouvry}
\address{Universit\'e Paris Sud, Laboratoire de Math\'ematique\\
  Campus d'Orsay\\ 91405 Orsay Cedex\\France}
\email{etienne.fouvry@math.u-psud.fr}

\author[E. Kowalski]{Emmanuel Kowalski}
\address{ETH Z\"urich -- D-MATH\\
  R\"amistrasse 101\\
  CH-8092 Z\"urich\\
  Switzerland} \email{kowalski@math.ethz.ch}

\author[Ph. Michel]{Philippe Michel}
\address{EPFL/SB/IMB/TAN, Station 8, CH-1015 Lausanne, Switzerland }
\email{philippe.michel@epfl.ch}

\thanks{Ph. M. was partially supported by the SNF (grant
  200021-137488) and the ERC (Advanced Research Grant 228304);
  \'E. F. thanks ETH Z\"urich, EPF Lausanne and the Institut
  Universitaire de France for financial support. }

\subjclass[2010]{11N25, 11N37, 11L05, 11T23}

\keywords{Ternary divisor function, arithmetic progressions, exponent
  of distribution, Voronoi formula, exponential sums over finite
  fields, trace functions, Kloostermania}

\begin{abstract}
  We show that the exponent of distribution of the ternary divisor
  function $d_3$ in arithmetic progressions to prime moduli is at
  least $1/2+1/46$, improving results of Heath-Brown and
  Friedlander--Iwaniec. Furthermore, when averaging over a fixed
  residue class, we prove that this exponent is increased to $1/2
  +1/34$.
\end{abstract}

\maketitle

\section{Introduction and statement of the main results}

For any positive integer $k\geq 1$, we denote by $d_{k}$ the $k$--fold
divisor function: for $n$ a positive integer, $d_{k}(n)$ is the number
of solutions of the equation
$$
n=n_{1}\dots n_{k},
$$
where the $n_{i}$ are positive integers. The purpose of this
paper is to investigate the exponent of distribution of the ternary
divisor function $d_{3}$ in arithmetic progressions. More generally,
we will say that a real number $\Theta>0$ is an \emph{exponent of
  distribution for $d_k$ restricted to a set $\mathcal{Q}$ of moduli}
if, for any $\eps>0$, for any $q\in\mathcal{Q}$ with $q\leq
x^{\Theta-\eps}$ and any residue class $a\bmod{q}$ with $(a,q)=1$, we
have a uniform asymptotic formula
\begin{equation}\label{eq-ued}
\sum_{\substack{n\equiv a\bmod q  \\n\leq x}} d_{k}(n)=
\frac{1}{\varphi (q)} \sum_{\substack{(n,q)=1 \\n\leq x}} d_{k}(n)
+O\Bigl(\frac{x}{q(\log x)^A}\Bigr)
\end{equation}
for any $A>0$ and $x\geq 2$, the implied constant depending on $A$ and
$\eps$ only. If $\mathcal{Q}$ contains all positive integers, we speak
only of \emph{exponent of distribution}.
\par
It is widely believed $\Theta=1$ is an exponent of distribution for
all $k$. This fact, if true, has deep consequences on our
understanding of the distribution of primes in arithmetic progressions
to very large moduli, going beyond the direct reach of the Generalized
Riemann Hypothesis. It is therefore not surprising that this problem
has been studied extensively, and that it is especially relevant to
obtain an exponent of distribution $\Theta>1/2$, since this goes
beyond the techniques involving the Bombieri--Vinogradov Theorem.
\par
As a consequence of the combinatorial structure of $d_k$ (essentially
by Dirichlet's hyperbola method in dimension $k$), one instantly
deduces that $\Theta =1/k$ is an exponent of distribution for $d_k$,
in particular $\Theta=1$ for $k=1$. It was noted by Linnik and Selberg
that for $k=2$ (the classical divisor function), a fairly direct
application of Weil's bound for Kloosterman sums yields $\Theta=2/3$.
\par
The only other case for which an exponent of distribution greater than
$1/2$ is known is for $d_3$: in their groundbreaking paper,
Friedlander and Iwaniec \cite{FrIw1}, showed that $\Theta = 1/2
+1/230$ is an exponent of distribution, a value later improved by
Heath--Brown to $\Theta = 1/2+1/82$ \cite{H-B}.  The proof of these
two results use deep applications of Deligne's proof of the Riemann
Hypothesis for algebraic varieties over finite fields. Our main result
is a further, rather significant, improvement in the case of prime
moduli.

\begin{theorem}\label{central} For every non-zero integer $a$, every $\eps,A>0$, every $x\geq 2$ and  every prime $q$, coprime with $a$, satisfying $$q\leq x^{\frac{1}2+\frac{1}{46}-\eps},$$ we have
$$
\sum_{\substack{n\equiv a\bmod q \\n\leq x}} d_{3}(n)
= \frac{1}{\varphi (q)} \sum_{\substack{(n,q)=1 \\n\leq x}} d_{3}(n)
 + O\Bigl( \frac{x}{q(\log x)^{A}}\Bigr),
$$
where the implied constant only depends on $\eps$ and $A$ (and not on $a$); in other terms, the value $\Theta = 1/2+1/46$ is an exponent of distribution for the divisor function $d_{3}$ restricted to prime moduli.
\end{theorem}

It is certainly possible to extend our arguments to composite
moduli. This would require some generalization of our main tools,
which are general estimates for sums of \emph{trace functions over
  finite fields} twisted by Fourier coefficients of Eisenstein series
(see Theorem \ref{d(n)K(n)} below).

\subsection{Distribution on average}  

In applications, estimates like \eqref{eq-ued} are often required only
{\em on average} over moduli $q\leq Q$ and it is no surprise that
sometimes these become available for $Q=x^{\theta}$ and $\theta$
larger than the known exponents of distribution. For instance, since
the function $d_{k}$ is multiplicative, the large sieve inequality
implies that \eqref{eq-ued} holds on average for any $\theta<1/2$
(see, e.g, \cite{Mo} or~\cite{Wo}). Concerning $d_3$, Heath--Brown
\cite[Theorem 2]{H-B} proved the following result (in a slightly
stronger form):

\begin{equation*} 
  \sum_{q\leq Q}\ \max_{y\leq x}\ \max_{(a,q)=1}\Bigl\vert\ 
  \sum_{\substack{n\equiv a\bmod q  \\n\leq y}} d_{3}(n) -
  \frac{1}{\varphi (q)} \sum_{\substack{(n,q)=1 \\n\leq y}} d_{3}(n)\ \Bigr\vert
  =O\bigl(x^{\frac{40}{51}+\varepsilon} Q^\frac{7}{17}\bigr),
\end{equation*}
which shows that~\eqref{eq-ued} holds on average for $q\leq
x^{\frac{11}{21}-\varepsilon}$.
\par
Although we can not improve this (on average over prime moduli), we
are able to improve Theorem~\ref{central} for $d_3$ on average over
prime moduli in a single residue class $n\equiv a\mods{q}$, where
$a\not=0$ is fixed.

\begin{theorem}\label{onaverage} 
  For every non--zero integer $a$, for every $\varepsilon >0$ and for
  every $A>0$, we have
$$
\sum_{q\leq x^{\frac{9}{17}-\varepsilon}\atop q\text{ prime}, q\nmid
  a} \Bigl\vert \sum_{\substack{n\equiv a\bmod q \\n\leq x}} d_{3}(n)
- \frac{1}{\varphi (q)} \sum_{\substack{(n,q)=1 \\n\leq x}} d_{3}(n)
\Bigr\vert = O\Bigl( \frac{x}{(\log x)^A}\Bigr),
$$
where the implied constant only depends on $(a,A,\eps)$.
\end{theorem}

\begin{remark}
  It is implicit from our proof and from the results of~\cite{BFI} on
  which it is based that this estimate holds uniformly for $1\leq
  |a|\leq x^{\delta}$, for some $\delta>0$ depending on $\eps$.
\end{remark}

\subsection{Remarks on the proofs}

The proof of Theorem~\ref{central} builds on two main ingredients
developed in \cite{FKM1} and~\cite{FKM2}:
\begin{enumerate}
\item A systematic exploitation of the spectral theory of modular
  forms; for instance, although our most important estimate involves
  only the divisor function, its proof passes through the full
  spectrum of the congruence subgroup $\Gamma_0(q)\subset\mathrm{SL}_2(\Zz)$;
\item The formalism of Frobenius trace functions modulo a prime, like
  the normalized Kloosterman sums $a\mapsto p^{-1/2}S(a,1;p)$: such
  functions are considered as fundamental building blocks in
  estimates, and not necessarily ``opened'' too quickly as exponential
  sums (for instance, the crucial estimate of a three-variable
  character sum in~\cite{FrIw1} is, in our treatment, hidden in the
  very general statement of Theorem~\ref{th-type-3}, which follows
  from~\cite{FKM2}.)
\end{enumerate}
\par
The outcome are two different estimates (Theorems \ref{d(n)K(n)} and
\ref{th-type-3}) which are applied through a simple combinatorial
decomposition of the main sum (compare, e.g., Section~\ref{sec-final}
with~\cite[\S 7]{H-B}).
\par
The proof of Theorem~\ref{onaverage} combines these estimates with the
``Kloostermaniac'' techniques pioneered by Deshouillers and Iwaniec
and pursued with great success by Bombieri, Fouvry, Friedlander and
Iwaniec to study primes in large arithmetic progressions.

\begin{remark}
  After the first version of this paper had been submitted for
  publication, the arithmetic importance of the exponent of
  distribution of the ternary divisor function for suitable large
  moduli was highlighted again in Zhang's groundbreaking
  work~\cite{zhang} on bounded gap between primes. Some of the
  techniques developped in the present paper have since been used --
  within the project Polymath 8 -- to give improvements of Zhang's
  results (see~\cite[Section 9]{Polymath} for a discussion).
\end{remark}

\subsection{Notation} 

We denote $e(z)=e^{2i\pi z}$ for $z\in\Cc$. For $n\geq 1$ and for an
integrable function $w\,:\, \Rr^n\rightarrow \Cc$, we denote by
\begin{equation*} 
  \fourier{w} (\xi)=\int_{\Rr^n} w(t) e(-\langle t,\xi\rangle)\, dt
\end{equation*}
its Fourier transform, where $\langle \cdot,\cdot\rangle$ is the
standard inner product on $\Rr^n$.
\par
If $q\geq 1$ is a positive integer and if $K\, :\ \Zz \longrightarrow
\Cc$ is a periodic function with period $q$, its Fourier transform is
the periodic function $\fourier{K}$ of period $q$ defined on $\Zz$ by
\begin{equation*} 
  \fourier{K}(n) = \frac{1}{\sqrt{q}}\sum_{h\bmod q} K(h) e\Bigl(
  \frac{hn}{q}\Bigr)
\end{equation*}
(note the minor inconsistency of sign choices). We have
$\widehat{\fourier{K}}(n)=K(-n)$ for all integer $n$.
\par
Given a prime $p$ and a residue class $a$ invertible modulo $p$, we
denote by $\bar{a}$ the inverse of $a$ modulo $p$.  For a prime $p$
and an integer $a$, the normalized hyper-Kloosterman sum
$\hypk_{k}(a;p)$ is given by
$$
\hypk_{k} (a;p):=\frac{1}{p^\frac{k-1}{2}}\
\underset{\substack{x_{1},\dots, x_{k}\bmod p\\ x_{1}\cdots x_{k}\equiv
    a \bmod p}} {\sum\quad \sum}\ e\Bigl( \frac{x_{1}+\cdots
  +x_{k}}{p}\Bigr).
$$
\par
The notation $q\sim Q$ means $Q< q \leq 2Q$, and $f(x)=O(g(x))$ for
$x\in X$ is synonymous with $f(x)\ll g(x)$ for $x\in X$.

\section{Summation formulas} 

\subsection{Poisson summation formula}

We recall a form of the Poisson summation formula in arithmetic
progressions:

\begin{lemma}\label{Poisson} 
  For any positive integer $q\geq 1$, any function $K$ defined on
  integers and $q$-periodic, and any smooth function $V$ compactly
  supported on $\Rr$, we have
$$
\sum_{n\geq 1}K(n)V(n)=\frac{1}{\sqrt{q}} \sum_{m}
\fourier{K}(m)\fourier{V}\Bigl(\frac{m}{q}\Bigr),
$$
and in particular
$$
\sum_{n\equiv a \bmod q}V(n)=\frac{1}{q} \sum_{m} e\Bigl(
\frac{am}{q}\Bigr)\fourier{V}\Bigl(\frac{m}{q}\Bigr).
$$
\end{lemma}
 
\subsection{The tempered Voronoi summation formula}
 
We will also make crucial use of a general (soft) version of the
classical summation formula of Voronoi for the divisor function $d_2$,
which goes back to Deshouillers and Iwaniec~\cite[Lemma 9.2]{DI}. This
formula is called the {\em tempered Voronoi summation formula} in
\cite[Prop. 4.11]{IK}, and amounts essentially to an application of
the Poisson formula in two variables $(x,y)$ to a function depending
on the product $xy$. 
\par
We define the \emph{Voronoi transform} $\bessel{K}$ of a $p$-periodic
function $K\,:\, \Zz\lra \Cc$ by
$$
\bessel{K}(n) = \frac{1}{\sqrt{p}}\sum_{\substack{h\bmod p\\(h,p) =1}}
\fourier{K}(h) e\Bigl( \frac{n\overline h}{p} \Bigr).
$$
\par
In other words, we have
$$
\bessel{K}(n)=
\begin{cases}
  \displaystyle{ \frac{1}{\sqrt{p}}\sum_{h_{1}h_{2}=n}
    \fourier{K}
    (h_{1}) e\Bigl( \frac{h_{2}}{p}\Bigr)},& \text{ if } p\nmid n,\\
  \\
  \displaystyle{K(0) -\frac{\fourier{K}(0)}{\sqrt{p}}},& \text{
    if } p\mid n.
 \end{cases}
$$

\begin{proposition}[Tempered Voronoi formula modulo
  primes]\label{Voronoigeneral0} 
  Let $p$ be a prime number, let $K\, :\ \Zz \longrightarrow \Cc$ be a
  $p$-periodic function, and let $G$ be a smooth function on $\Rr^2$
  with compact support. We have
$$
\sum_{m,n\in\Zz} K(mn) G(m,n)=\frac{\fourier{K}(0)}{\sqrt{p}}
\sum_{m,n\in\Zz} G(m,n)+
\frac{1}{p}\sum_{m,n} \bessel{K}(mn)\fourier{G}\Bigl(\frac{m}p,\frac
np\Bigr).
$$
\end{proposition}
 
\begin{proof}
  We have the discrete inversion formula
$$
K(mn) = \frac{1}{\sqrt{p}}\fourier{K}(0)
+\frac{1}{\sqrt{p}}\sum_{(h,p)=1} \fourier{K}(h)
e\Big(-\frac{hmn}{p}\Bigr),
$$ 
and then for any integer $a$ coprime to $p$, the tempered Voronoi
formula of~\cite[Prop. 4.11]{IK} gives
$$
\sum_{m,n\in\Zz}G(m,n)e\Bigl(\frac{amn}p\Bigr)
=\frac1p\sum_{m,n\in\Zz}\fourier{G}\Bigl(\frac mp,\frac
np\Bigr)e\Bigl(-\frac{\ov amn}p\Bigr),
$$
so that the result follows by multiplying this by $\fourier{K}(-a)$,
summing over $(a,p)=1$.
\end{proof}

\subsection{The combined formula}

We now combine the Poisson formula and the Voronoi formula to give a
general transformation formula for three-variable sums.


\begin{corollary}[Poisson-Voronoi formula]\label{cor-combined}
  Let $\uple{V}=(V_1,V_2,V_3)$ where $V_i$ are smooth functions with
  compact support in $]0,+\infty[$. Let $p$ be a prime number, and let
  $K$ be a $p$-periodic function on $\Zz$, supported on integers
  coprime to $p$. Define
$$
S(\uple{V};p,K)=\sum_{m_1,m_2,m_3\geq
  1}V_1(m_1)V_2(m_2)V_3(m_3)K(m_1m_2m_3).
$$
\par
We then have
$$
S(\uple{V};p,K)=\tA+\tB+\tC+\tD
$$
where
\begin{gather*}
  \tA=\frac{\fourier{K}(0)}{\sqrt{p}} \sumtriple_\stacksum{p\nmid
    m_1m_2}{m_3\geq 1} V_{1}(m_1)V_2(m_2)V_3(m_3),
  \\
  \tB=-\frac{\fourier{K}(0)}{p^{3/2}} \sumtriple_\stacksum{m_1,m_2\geq
    1}{p\nmid n_3}V_1(m_1) V_2(m_2)\fourier{V}_3
  \Bigl(\frac{n_3}{p}\Bigr),
\end{gather*}
\begin{gather*}
  \tC=\frac{\fourier{K}(0)}{p^2} \Bigl\{ \fourier{V}_1
  (0)\sum_{n_2}\fourier V_2\Bigl(\frac{n_2}{p}\Bigr)+\fourier{V}_2
  (0)\sum_{n_1}\fourier V_1\Bigl(\frac{n_1}{p}\Bigr)-\widehat
  V_1(0)\fourier{V}_2(0)\Bigr\}
  \sum_{p\nmid n_3}\fourier{V}_3\Bigl(\frac{n_3}{p}\Bigr),\\
  \\
  \tD=\frac{1}{p^{3/2}}\sumtriple_{\substack{n_1n_2\not=0\\p\nmid n_3}}
  \fourier{V}_{1}\Bigl(\frac{n_1}{p}\Bigr)\fourier{V}_2
  \Bigl(\frac{n_2}{p}\Bigr) \fourier{V}_3\Bigl(\frac{n_3}{p}\Bigr)
  \bbessel{K}(n_1n_2,n_3),
\end{gather*}
with
\begin{equation}\label{eq-fourier-bessel}
  \bbessel{K}(x,n)=\frac{1}{\sqrt{p}}
  \sum_{y\in \Fpt}\fourier{K}(n\bar{y})\hypk_2(xy;p).
\end{equation}
\par
In the sums above, $m_1,m_2,m_3$ run over integers $\geq 1$,
with additional conditions, while $n_1,n_2,n_3$ run over all
integers in $\Zz$.
\end{corollary}

\begin{proof}
  We start by applying the Poisson formula (Lemma \ref{Poisson}) to
  the variable $m_3$. Denoting by $n_3\in\Zz$ the dual Fourier
  variable, we isolate the terms with $n_3\equiv 0\bmod p$ and obtain
\begin{multline*} 
  S(\uple{V};p,a) =
  \frac{1}{\sqrt{p}}\sumtriple_\stacksum{(m_1m_2,p)=1}{n_3\equiv
    0\bmod{p}}
  V_{1}(m_1)V_2(m_2)\fourier{V}_3\Bigl(\frac{n_3}{p}\Bigr)
  \fourier{K}(0)
  \\
  +\frac{1}{\sqrt{p}}\sumtriple_\stacksum{(m_1m_2,p)=1}{(n_3,p)=1}
  V_{1}(m_1)V_2(m_2)\widehat V_{3}\Bigl(\frac{n_3}{p}\Bigr)
  \fourier{K}(n_3\overline{m_1m_2}),
\end{multline*}
where $\overline {m_1m_2}$ is the multiplicative inverse of $m_1m_2$
modulo $p$.
\par
We use again the Poisson formula to transform backward the first sum,
and get
$$
\sum_{n_3\equiv
  0\bmod{p}}\fourier{V}_3\Bigl(\frac{n_3}{p}\Bigr)=\sum_{n_3}\fourier{V}_3
(n_3)=\sum_{m_3}V_3(m_3)
$$ 
so that this first term is equal to the quantity $\tA$ of the
statement. 
\par
We now consider the second sum, which we denote
$\Sigma(\uple{V};p,K)$.  We apply the tempered Voronoi summation
formula of Proposition~\ref{Voronoigeneral0} to the sum over $m_1$ and
$m_2$, and to the function
$$
m\mapsto L(m,n_3) =\fourier{K}(n_3\bar{m})\quad\text{for}\quad p\nmid m,
$$
extended by $0$ to the $m$ divisible by $p$. Denoting by $\fourier{L}$
and $\bessel{L}$ the corresponding transforms with respect to the
variable $m$ when $n_3$ is fixed, we note that
\begin{equation}\label{567}
\fourier{L}(0,n_3)=\frac{1}{\sqrt{p}}\sum_{x\in\Fpt}\fourier{K}(x)=
-\frac{1}{\sqrt{p}}\fourier{K}(0)
\end{equation}
for every $n_3$ coprime to $p$, since $K(0)=0$ by assumption.
\par
Thus we obtain
$$
\Sigma(\uple{V};p,K)=\Sigma_1(\uple{V};p,K)+\Sigma_2(\uple{V};p,K)
$$
where
$$
\Sigma_1(\uple{V};p,K)=-\frac{\hat{K}(0)}{p^\frac{3}{2}}
\, \sumtriple_\stacksum{m_1,m_2\geq 1}{p\nmid n_3}V_1(m_1) V_2(m_2)\widehat
V_{3}\Bigl(\frac{n_3}{p}\Bigr)=\tB,
$$ 
and
\begin{multline*}
  \Sigma_2(\uple{V};p,K)=\frac{1}{p^{3/2}}\sum_{p\nmid n_3}
  \fourier{V}_3\Bigl(\frac{n_3}{p}\Bigr)\Bigl\{\bessel{L}(0,n_3)
  \sumsum_{n_1n_2=0} \fourier{V}_{1}\Bigl(\frac{n_1}{p}\Bigr)
  \fourier{V}_2
  \Bigl(\frac{n_2}{p}\Bigr) \\
  +\sumsum_{n_1n_2\not=0}\fourier{V}_1\Bigl(\frac{n_1}{p}\Bigr)
  \fourier{V}_2\Bigl(\frac{n_2}{p}\Bigr)
\bessel{L}(n_1n_2,n_3)\Bigr\}.
\end{multline*}
\par
A straightforward computation shows that
$$
\bessel{L}(x,n_3)=\frac{1}{\sqrt{p}}
\sum_{y\in\Fpt}\fourier{K}(n_3\bar{y})\hypk_2(xy;p)=\bbessel{K}(x,n_3).
$$
\par
In particular, we have
$$
\bessel{L}(0,n_3)=-\frac{1}{\sqrt{p}}\sum_{y\in\Fpt}\fourier{K}(y)=
-\frac{\fourier{L}(0,n_3)}{\sqrt p},
$$
so, by \eqref{567},  the first term in $\Sigma_2(\uple{V};p,K)$ is 
$$
\frac{\fourier{K}(0)}{p^2} \Bigl\{ \widehat V_1(0)\sum_{n_2}\fourier
V_2\Bigl(\frac{n_2}{p}\Bigr)+\widehat V_2(0)\sum_{n_1}\fourier
V_1\Bigl(\frac{n_1}{p}\Bigr)-\widehat V_1(0)\widehat V_2(0)\Bigr\}
\sum_{p\nmid n_3}\fourier{V}_3 \Bigl(\frac{n_3}{p}\Bigr)=\tC,
$$ 
while the remaining contribution is the quantity $\tD$.
\end{proof}

In this paper, we will only need the following case of these
transformations: 

\begin{lemma}\label{1060}
  Let $p$ be a prime, let $a$ be an invertible residue class modulo
  $p$, and let, for $n$ integer
$$
K(n)=\delta_a(n).
$$
Then, for every $n$ not divisible by $p$ and for every $x$, we have
the equality
$$
\bbessel{K}(x,n)=\frac{1}{\sqrt{p}} \hypk_3(anx;p).
$$
\end{lemma}

\begin{proof}  
  Obviously, we have $\hat{K}(x)=\frac{1}{\sqrt{p}}e(ax/p)$, and the
  result then follows from the definition \eqref{eq-fourier-bessel}
  after opening the Kloosterman sum.
\end{proof}

\section{Results on trace functions}\label{traceresults}

The key new input to the present paper is the application of a special
case of the following very general theorem concerning algebraic trace
functions summed against the divisor functions.

\begin{theorem}[Divisor twists of trace functions]\label{d(n)K(n)}
  Let $p$ be a prime number, and let $K$ be the trace function of an
  $\ell$-adic middle-extension sheaf $\sheaf{F}$, pointwise of weight
  $0$, on the affine line over $\Fp$. Assume that $\sheaf{F}$ is
  geometrically irreducible and is not geometrically isomorphic to an
  Artin-Schreier sheaf associated to an additive character modulo $p$.
\par
Let $Q\geq 1$ and let $V, W$ be smooth test functions, compactly
supported in $[1/2,2]$, such that, for $\xi >0$, one has 
\begin{equation}\label{condV}
\xi^j V^{(j)} (\xi),\ \xi^j W^{(j)} (\xi) \ll Q^j,
\end{equation}
for all integer $j\geq 0$, with implicit constants that depend on
$j$. For any $M_1,M_2\geq 1$, we have
$$
\sum_{m_1,m_2\geq 1} K(m_1m_2)
V\Bigl(\frac{m_1}{M_1}\Bigr)W\Bigl(\frac{m_2}{M_2}\Bigr) \ll QM_1M_2
\Bigl( 1 +\frac{p}{M_1M_2}\Bigr)^{1/2} p^{-\eta},
$$
for any $\eta <1/8$. The implicit constant depends only on $\eta$, on
the implicit constants in~\emph{(\ref{condV})} and on the conductor of
$\sheaf{F}$.
\end{theorem}

This is Theorem 1.15 in~\cite{FKM2}, which depends essentially on
methods of~\cite{FKM1}, to which we refer for more details and
definitions concerning trace functions.  For the purpose of this
paper, it is sufficient to know that for any $k\geq 2$, any prime $p$
and $h\in\Fpt$, the functions given by
\begin{equation}\label{eq-special-weight}
  K(a)=(-1)^{k-1}\hypk_k(ah;p),\text{ for } a\in\Fpt, \quad\quad K(0)=
  (-1)^k p^{-(k-1)/2}
\end{equation}
are trace functions associated to geometrically irreducible sheaves
$\sheaf{F}_{k,h}$ of rank $k$ with conductor bounded by a constant
$C_k$ depending only on $k$, which is proved
in~\cite[Prop. 10.3]{FKM1}. In fact, only the case $k=3$ will be used.
\par
Another general result is the following estimate for general ``type
$III$'' sums, which follows from our results in~\cite{FKM2}. In the
context of the function $d_3$, the corresponding trick of grouping
variables appears in the work of Heath-Brown
(see~\cite[p. 42--43]{H-B}, where previous occurrences in work of Y\"uh
is mentioned).

\begin{theorem}\label{th-type-3}
  Let $p$ be a prime, and let $K$ be the trace function of an
  $\ell$-adic middle-extension sheaf $\sheaf{F}$, pointwise of weight
  $0$, on the affine line over $\Fp$. Assume that $\sheaf{F}$ is
  geometrically irreducible and is not geometrically isomorphic to a
  tensor product of an Artin-Schreier sheaf associated to an additive
  character modulo $p$ and a multiplicative Kummer sheaf.
\par
For any complex coefficients $(\alpha(n))_{\vert n \vert \leq N_1}$,
$(\beta(n))_{|n|\leq N_2}$, $(\gamma(n))_{|n|\leq N_3}$ with modulus
less than $1$ and any $\eps>0$, we have
\begin{multline*}
  \sumtriple_{\substack{1\leq |n_i|\leq N_i\\p\nmid n_3}}
  \alpha(n_1)\beta(n_2)\gamma(n_3)K(n_1n_2n_3)\\
   \ll
 (\log
  p)^{1/2}  (N_1N_2N_3)^{1/2+\eps}  \Bigl(
  \frac{N_1N_2N_3}{\sqrt{p}}+N_1N_2+N_3\sqrt{p}\Bigr)^{1/2},
\end{multline*}
where the implied constant depends only on $\eps>0$ and polynomially
on $\cond(\sheaf{F})$.
\end{theorem}

\begin{proof}
  After elementary dyadic subdivisions (and summing over the separate
  signs), we see that it is enough to apply~\cite[Th. 1.16 (1)]{FKM2}
  with the choices
\begin{gather*}
  M=N_3,\quad\quad N=N_1N_2,\\
  \alpha_m=\gamma(m),\quad\quad \beta_n=(\alpha\star \beta)(n)
\end{gather*}
where $\star$ is the Dirichlet convolution. The bound we derive
from~\cite{FKM2} is 
$$
\Bigl(\sum_{m}|\gamma(m)|^2\Bigr)^{1/2}
\Bigl(\sum_{n}|(\alpha\star\beta)(n)|^2\Bigr)^{1/2}
(N_1N_2N_3)^{1/2}
\Bigl(\frac{1}{p^{1/4}}+\frac{1}{\sqrt{N_3}}+\frac{p^{1/4}(\log
  p)^{1/2}}
{\sqrt{N_1N_2}}\Bigr),
$$
and one checks easily that this implies the statement above.
\end{proof}

Again we will only need to know that we can apply this to the
functions $K$ above.

\section{Preliminary reductions}\label{prelred}

In this section, we will set up the proof of Theorem~\ref{central}, in
a way very similar to the preliminaries in~\cite{FrIw1}
and~\cite{H-B}. The notational conventions that we introduce here will
be valid throughout the remainder of the paper.
\par
In \S \ref{prelred} and in \S \ref{Conc} the letter $q$ is reserved to
denote a prime number, $x\geq 1$ is a real number, and we denote $\LL
=\log 2x$ for simplicity. We define
$$
 S(x;q,a):=\sum_{\substack{ n\equiv a \bmod q\\ n\leq x}}
 d_{3}(n)=\sum_{\substack{ m_1m_2m_3\equiv a \bmod q\\ m_1m_2m_3\leq
     x}} 1 ,
$$
where $a$ is some integer coprime with $q$, and
$$
S^*(x;q):= \sum_{\substack{n\leq x\\ (n,q)=1}} d_3( n),\quad\quad S(x)
=\sum_{n\leq x} d_3 (n).
$$
\par
If $q<x^{1/100}$, we have~(\ref{eq-ued}) trivially. Hence we can
assume that
\begin{equation}\label{q>x1/2}
  x^\frac{1}{100}\leq q \leq x^{\frac{99}{100}}.
\end{equation}
\par
Since $q$ is prime, this assumption \eqref{q>x1/2} implies 
\begin{equation}\label{S=S*}
S^*(x;q) =S(x) +O_\epsilon (x^{\frac{99}{100} +\epsilon}),
\end{equation}
for every $\epsilon >0$. Moreover, $S(x)$ is of size $\frac{1}{2}
x\LL^2$, and hence Theorem \ref{central} will follow if we prove that,
for any $\theta<1/2+1/46$, we have
\begin{equation}\label{reduction}
  S(x;q,a)=\frac{1}{q} S(x)+ O\Bigl(\frac{x}{q\,\LL^A}\Bigr),
\end{equation}
for any $A>0$, uniformly for $a$ not divisible by $q$ such that
$x^{1/100}\leq q\leq x^{\theta}$, the implied constant depending on
$\theta$ and $A$.

We will need to make the three variables $m_1,m_2$ and $m_3$
independent and smooth. For this purpose, we use a smooth partition of
unity, which is given by the following lemma (see \cite[Lemme 2]{Fo85}
for instance).

\begin{lemma}\label{2.1}
  For every $\Delta >1$, there exists a sequence $(b_{\ell,
    \Delta})_{\ell\geq 0}$ of smooth functions with support included
  in $[\Delta^{\ell -1}, \Delta^{\ell +1}]$, such that
$$
\sum_{\ell =0}^\infty b_{\ell, \Delta}(\xi)=1\text{ for all } \xi \geq
1,
$$
and
\begin{equation}\label{decay}
b_{\ell, \Delta}^{(\nu)} (\xi) \ll_{\nu} \xi^{-\nu} \Delta^\nu (\Delta
-1)^{-\nu}, \text{ for all } \xi \geq 1 \text{ and } \nu \geq 0.
\end{equation}
 \end{lemma}
We take $\Delta$  slightly larger than $1$, namely
\begin{equation*} 
\Delta = 1+ \LL^{-B}
\end{equation*}
for some parameter $B\geq 1$.
\par
From now on, we denote by $M_i$, $1\leq i\leq 3$, some parameters of
the form 
\begin{equation}\label{eq-mi}
  M_i=\Delta^{\ell}=(1+\LL^{-B})^{\ell},
\end{equation}
where $\ell\geq 0$ is an integer. For such a variable
$M_i=\Delta^{\ell}$, we define
\begin{equation}\label{defVi}
V_{i}(t) = b_{\ell, \Delta}(t),
\end{equation}
where $b_{\ell, \Delta}$ are the functions given by Lemma
\ref{2.1}. Thus, the derivatives of $V_{i}$ satisfy
\begin{equation}\label{boundderiv}
  V_{i}^{(\nu)}(t) \ll_{\nu} t^{-\nu} \LL^{B\nu}.
\end{equation}
\par
The bound \eqref{boundderiv} implies the classical fact that $\widehat
V_{i}(\xi)$ decays quickly, namely
\begin{equation}\label{fourierdecay}
  \widehat V_{i}(\xi)\ll_{\nu} M_i \Bigl(\frac{\LL^B}{|\xi|M_i}\Bigr)^\nu,
\end{equation}
for all integers $\nu \geq 0$ and $\xi \not= 0$.
\par
For $\uple{M}=(M_1,M_2,M_3)$, we can now consider the smooth sums
\begin{equation}\label{defS}
  S(\uple{M}; q,a)=
  \underset{\substack{m_1m_2m_3 \equiv a \bmod q}}{\sum\ \sum\ \sum} 
  V_{1}(m_1)V_2(m_2) V_{3}(m_3),
\end{equation}
and 
\begin{equation}\label{defSS}
  S(\uple{M}) 
  =\underset{  m_1,m_2,m_3}{\sum\ \sum \ \sum}\, V_{1}(m_1) V_{2}(m_2)V_3(m_3).
\end{equation}
\par
Our preparation for Theorem~\ref{central} is given by the following
lemma:

\begin{lemma}\label{964}
For any $A>0$, we can select $B\geq 1$ such that
$$
S(x,q;a)-\frac{1}{q}S(x)=\sum_{\uple{M}}\Bigl(S(\uple{M};q,a)-\frac{1}{q}S(\uple{M})\Bigr)
+O_{\theta}(q^{-1}x\LL^{-A}),
$$
where $\uple{M}=(M_1,M_2,M_3)$ runs over triples of $M_i$ as above
such that
\begin{equation}\label{MN}
  x\LL^{-B} \leq M_1M_2M_3\leq x.
\end{equation}
\end{lemma}

\begin{proof}
Using the partition of unity above, we have
\begin{align}\label{split}
  S(x;q,a)&= \underset{\uple{M}=(M_1,M_2,M_3)}{\sum\ \sum\ 
    \sum}S(\uple{M};q,a) +
  O\Bigl(\sum_{\substack{x\leq n \leq x\Delta^3\\
      n\equiv a \bmod q}} d_{3}(n)\Bigr)\nonumber\\
  &= \underset{\uple{M}=(M_1,M_2,M_3)}{\sum\  \sum\  \sum}S(\uple{M};q,a) +
  O(xq^{-1} \LL^{2-B}),
\end{align}
where the sum ranges over all the triples $\uple{M}=(M_1,M_2,M_3)$ of
the form above such that $M_1M_2M_3\leq x$, and the bound on the error
term is based on a classical estimate for the sum of the divisor
function in arithmetic progressions, restricted to an interval (see
\cite[Th. 2]{Sh} for instance).
\par
Similarly, the contribution to this sum of the triples $(M_1,M_2,M_3)$
satisfying $ M_1M_2M_3 \leq x\LL^{-B}$ satisfies
\begin{equation}\label{Err2}
  \underset{M_1M_2M_3\leq x\LL^{-B}}{\sum\ \sum\ 
    \sum}S(\uple{M};q,a)    \leq \sum_{\substack{1\leq n \leq 2 x\LL^{-B}\\ 
      n\equiv a \bmod q}}d_{3}(n)\ll  xq^{-1}  \LL^{2-B}.
\end{equation}
\par
Thus by selecting $B=B(A)$ large enough in \eqref{split} and
\eqref{Err2}, we get
\begin{equation*} 
  S(x;q,a)=\underset{(M_1,M_2,M_3)}{\sum\  \sum\  \sum}\ S(\uple{M};q,a) +
  O\bigl(xq^{-1} \LL^{-A}),
\end{equation*}
where the sum is over the triples $(M_1,M_2,M_3)$ such that~(\ref{MN})
holds. A similar result holds for the sum $S(x)$, and gives the
result.
\end{proof}

Due to the symmetry of the problem, it is natural to introduce the
following condition
\begin{equation}\label{decreasing}
M_3 \geq M_2 \geq M_1.
\end{equation}
\par
Since the number of triples $\uple{M}$ satisfying~(\ref{MN}) with
$M_i$ of the form~(\ref{eq-mi}) is $\ll \LL^{3B+3}$, Lemma \ref{964}
shows that~\eqref{reduction} (and hence Theorem \ref{central}) will
follow if we can show that for any $\theta<1/2+1/46$ and $A>0$, we
have
\begin{equation}\label{271}
  S(\uple{M};q,a) =\frac{1}{q} S(\uple{M}) + O_\theta (q^{-1} x \LL^{-A}),
\end{equation}
uniformly for all triples $\uple{M}=(M_1,M_2,M_3) $ satisfying
\eqref{MN}, \eqref{decreasing} and \eqref{eq-mi} and for all integers
$a$ coprime with $q$ satisfying $x^\frac{1}{100}\leq q \leq
x^\theta$. The proof of this is the object of the next section.

\section{Conclusion of the proof of Theorem \ref{central}}
\label{Conc}

The first two subsections below establish estimates for
$S(\uple{M};q,a)$ which are non-trivial in two different ranges,
depending on the sizes of $M_1$, $M_2$, $M_3$. In the last subsection,
we combine them to derive~(\ref{271}).
\par
In order to present cleanly the two cases, we introduce the parameters
$\kappa$ and $\mu_i$ defined by
\begin{equation}\label{exponents}
  q=x^\kappa\text{ and } 
  M_i=x^{\mu_i}\text{ for } 1\leq i \leq 3, 
\end{equation}
so that $\kappa$ and  $\mu_i$ satisfy
\begin{equation*} 
  1/100\leq \kappa \leq 99/100,
\end{equation*}
and
\begin{equation}\label{exponentsmui}
  1-B\frac{\log\LL}{\LL}\leq
  \mu_1+\mu_2+\mu_3\leq 1, \, \mu_{3} \geq \mu_2 \geq  \mu_{1}\geq 0,
\end{equation}
as a consequence of \eqref{q>x1/2}, \eqref{MN} and \eqref{decreasing}.
We also remind the reader that $q$ denotes a prime number.

\subsection{Applying the combined summation formula}

We apply the combined summation formula of
Corollary~\ref{cor-combined} to $S(\uple{M};q,a)$, which is of the
form treated there with $K(n)$ the characteristic function of the
residue class $a\bmod{q}$. We then have
$$
\fourier{K}(0)=\frac{1}{\sqrt{q}},
$$
and, for $(q,n)=1$,
$$
\bbessel{K}(x,n)=\frac{\hypk_3(anx;q)}{\sqrt{q}},
$$
by Lemma \ref{1060}.  We therefore get the equality
\begin{equation}\label{decompoS}
  S(\uple{M};q,a)=\tA+\tB+\tC+\tD,
\end{equation} 
as in Corollary~\ref{cor-combined}, and we proceed to handle these
four terms.
\par
First of all, we have
\begin{equation}\label{decompoA}
  \tA=\frac{1}{q} \sumtriple_\stacksum{(m_1m_2,q)=1}{m_3\geq 1}
  V_{1}(m_1)V_2(m_2)V_3(m_3)=
  \frac{1}{q}S(\uple{M})+O\Bigl(\frac{x}{q^2}\Bigr),
\end{equation} 
which represents the desired main term. We will now find conditions
which ensure that $\tB$, $\tC$ and $\tD$ are small. We will use the
inequality 
\begin{equation}\label{often}
\fourier{V}_{i}(t)\ll M_{i},
\end{equation}
several times (see \eqref{fourierdecay}).
\par
First, we have
\begin{equation}\label{B=}
  \tB=-\frac1{q^{2}}
  \sumtriple_\stacksum{m_1,m_2, n_3}{(n_3,q)=1}V_1(m_1) V_2(m_2)\widehat
  V_{3}\Bigl(\frac{n_3}{q}\Bigr),
\end{equation}
and by applying twice Lemma~\ref{Poisson}, we get
\begin{align}
  \sum_{(n_3,q)=1} \widehat V_{3}\Bigl(\frac{n_3}{q}\Bigr) &=
  \sum_{n_3}\widehat V_{3}\Bigl(\frac{n_3}{q}\Bigr) - \sum_{q\mid n_3}
  \widehat
  V_{3}\Bigl(\frac{n_3}{q}\Bigr) \nonumber\\
  &= q \sum_{t\equiv 0 \bmod q} V_3(t) -\sum_t V_3(t) \ll
  M_3,\label{V3}
\end{align}
by the properties of the function $V_3$. Inserting this bound in
\eqref{B=} and combining with \eqref{MN}, we deduce
\begin{equation}\label{Bll}
  \tB\ll q^{-2} x.
\end{equation}
\par
Similarly, using the definition
$$
\tC=\frac{1}{q^\frac{5}{2}} \Bigl\{ \widehat V_1(0)\sum_{n_2}\fourier
V_2\Bigl(\frac{n_2}{q}\Bigr)+\widehat V_2(0)\sum_{n_1}\fourier
V_1\Bigl(\frac{n_1}{q}\Bigr)-\widehat V_1(0)\widehat V_2(0)\Bigr\}
\times\sum_{(n_3,q)=1}\wihat V_3\Bigl(\frac{n_3}q\Bigr),
$$ 
a computation similar to \eqref{V3} leads to 
\begin{equation}\label{Cll}
\tC\ll q^{-\frac{5}{2}} x.
\end{equation}
\par
We must now only deal with $\tD$. By Lemma \ref{1060}, we can write
$$
\tD=\frac{1}{q^{2}}\sumtriple_{\substack{n_1n_2\not=0\\(n_3,q)=1}}
\fourier{V}_{1}\Bigl(\frac{n_1}q\Bigr)\fourier{V}_2
\Bigl(\frac{n_2}q\Bigr) \fourier{V}_3\Bigl(\frac{n_3}{q}\Bigr)
\hypk_3(an_1n_2n_3;q).
$$
\par
For fixed $n_3$, the sum over $n_1$ and $n_2$ can be handled using
Theorem~\ref{d(n)K(n)}, according to the remark
after~(\ref{eq-special-weight}), except that the Fourier transforms of
the functions $V_i$ are not compactly supported. To handle this minor
difficulty, we use again a partition of unity. Precisely, we apply
Lemma \ref{2.1} with parameter $\Delta =2$, deriving a decomposition
$$
\tD=\frac{1}{q^2}\sum_{\uple{N}} \mathcal{D}(\uple{N})
$$
where $\uple{N}$ runs over triples $\uple{N}=(N_1,N_2,N_3)$, $N_i$ are
integers of the form $2^{\ell}$ for some $\ell\geq 0$, and
$$
\mathcal{D}(\uple{N})= 
\sumtriple_{\substack{n_1n_2\not=0\\(n_3,q)=1}} \Bigl(
\fourier{V}_{1}\Bigl(\frac{n_1}q\Bigr)W_1 (n_1)\Bigl)\,
\Bigr(\fourier{V}_2 \Bigl(\frac{n_2}q\Bigr) W_2 (n_2)\Bigl)\, \Bigr(
\fourier{V}_3\Bigl(\frac{n_3}{q}\Bigr) W_3(n_3)\Bigl)
\hypk_3(an_1n_2n_3;q)
$$
where $W_j(t)=b_{\ell,2}(t)$, a smooth function supported in
$[N_j/2,N_j]$.
\par
The inequality \eqref{fourierdecay} implies that the coefficients $n_i
\mapsto \fourier{V}_i(n_i/q)$ decay quickly as soon as
\begin{equation*} 
  n_i>\tilde N_i= qM_i^{-1}x^{\eta},
\end{equation*}
where $\eta>0$ is arbitrary small.  Thus we get
\begin{equation}\label{1134}
  \tD=\frac{1}{q^2} 
  \sum_{\stacksum{\uple{N}}{N_i\leq \tilde N_i}} {\mathcal
    D}(\uple{N}) 
  +O_{\eta} (x^{-1}).
\end{equation}
\par
The sum over $\uple{N}$ contains $\ll \LL ^3$ terms. By this remark
and by the relations \eqref{q>x1/2}, \eqref{decompoS}, 
\eqref{decompoA}, \eqref{Bll}, \eqref{Cll} and \eqref{1134}, we see that it is
enough (in order to prove \eqref{271}) to show that
\begin{equation}\label{final?}  
  \mathcal{D}(\uple{N}) \ll_{\epsilon,
    A} q x\LL^{-A},
\end{equation}
for all $\epsilon >0$, all $A>0$, all $\uple{M}$ satisfying
\eqref{exponents} and \eqref{exponentsmui}, all $N_{i}\leq
\tilde{N}_{i}$ and all $q=x^\kappa$ where $1/100\leq \kappa \leq
12/23-\epsilon$.
\par
We apply Theorem \ref{d(n)K(n)} to the sum over $(n_1,n_2)$ in
$\mathcal{D}(\uple{N})$. This means that, in that result, we take
parameters
\begin{gather*}
(M_1,M_2)=(N_1,N_2),\quad\quad
K(n)=\hypk_{3}(an_3n;q)\\
V(x)=M_1^{-1}\fourier{V}_{1}(xN_1/q) W_{1}(xN_1), \quad
W(x)=M_2^{-1}\fourier{V}_{2}(xN_2/q) W_{2}(xN_2),
\end{gather*}
which ensure that~(\ref{condV}) holds with $Q=x^{2\eta}$, and we must
multiply the resulting bound by $M_1M_2$. 
\par
Since, in addition, we have already observed that the conductor of
$n\mapsto\hypk_{3}(an_3n;q)$ is bounded by an absolute constant, we
obtain the upper bound
\begin{equation*}
  \mathcal{D}(\uple{N})\ll_\eta M_1M_2M_3N_1N_2N_3
  \Bigl(1+\frac{q}{N_1N_2}\Bigr)^{1/2}q^{-\frac{1}{8}}x^{3\eta}
\end{equation*}
after applying Theorem \ref{d(n)K(n)} and summing trivially over
$n_{3}$.
\par
This bound is worst when $N_{i}=\tilde
N_{i}=qM_i^{-1}x^{\eta}$. Hence, using \eqref{exponentsmui}, this
implies
\begin{equation*}
 \mathcal{D}(\uple{N})\ll_\eta 
  \Bigl(1+\frac{x}{qM_{3}}\Bigr)^{1/2}q^{\frac{23}{8}}x^{6\eta}.
\end{equation*}
\par
It follows easily that \eqref{final?} is satisfied as soon as
\begin{equation}\label{coffee}
  \kappa\leq \frac{8}{15} - 4\eta\quad  \text{ and }\quad
  \mu_{3}\geq \frac{11}{4} \kappa -1 + 14 \eta.
\end{equation}
\par
This is our first estimate.
 
\subsection{Grouping variables}\label{thelastrick1}

In the totally symmetric situation where
$$
\mu_1=\mu_2=\mu_3=1/3
$$
the inequalities \eqref{coffee} are very restrictive and do not allow
to extend the value of the exponent of distribution beyond $1/2$.
Instead, we use Theorem \ref{th-type-3} (which builds on the
construction of a long variable by grouping two short ones).
\par
We obtain (see again~\eqref{often})
\begin{equation*} 
  \mathcal{D} (\uple{N})
  \ll_{\eta } (M_{1}M_{2}M_{3}) (N_{1}N_{2}N_{3})^\frac{1}{2}
  \Bigl(q^{-\frac{1}{2}} N_{1}N_{2}N_{3} +N_{1}N_{2}+q^\frac{1}{2}N_{3}
  \Bigr)^\frac{1}{2} x^\eta.
\end{equation*}
\par
The right-hand side is a non-decreasing function of the parameters
$N_{i}\leq \tilde N_{i}$, and it leads to
\begin{align*} 
  \mathcal{D} (\uple{N})&\ll_{\eta } x\cdot (q^3/x)^\frac{1}{2}
  \Bigl(q^{-\frac{1}{2}} (q^3/x) +M_{3}(q^2/x) +q^\frac{3}{2}
  M_{3}^{-1}
  \Bigr)^\frac{1}{2} x^{5\eta }\nonumber \\
  &\ll \Bigl( q^\frac{11}{4} + q^\frac{5}{2} M_{3}^\frac{1}{2}
  +q^\frac{9}{4}x^\frac{1}{2}M_{3}^{-\frac{1}{2}}\Bigr) x^{5\eta}.
\end{align*}
\par
This implies that \eqref{final?} is also satisfied when we have 
\begin{equation}\label{ineq}
  \kappa \leq  \frac{4}{7}-\eta, \quad\quad
  \frac{5}{2} \kappa -1 + 12 \eta  \leq \mu_{3}\leq  2-3 \kappa-12 \eta .
\end{equation}
 
\subsection{End of the proof of Theorem \ref{central}}
\label{sec-final}

For the final step, we combine the results of the last two
subsections. Choosing $\eta =\epsilon /10$ for $\epsilon>0$
very small, we see that whenever
$$
\kappa \leq 1/2 +1/46-\epsilon,
$$
we have
$$
\frac{11}{4} \kappa -1 + 14\eta   \leq  2-3 \kappa -12 \eta.
$$
\par
Looking at the conditions in \eqref{coffee} and \eqref{ineq}, we see
that the bound~\eqref{final?} holds provided that
\begin{equation*}
\mu_{3}\geq  \frac{5}{2} \kappa - 1 + 2\epsilon.
\end{equation*}
\par
But by \eqref{exponentsmui}, we have
$$
\mu_{3}\geq \frac{1}{3}- \frac{B\log \LL}{3 \LL}\geq \frac{5}{2}
\kappa - 1 + 2\epsilon
$$
for $x$ large enough. This completes the proof of Theorem
\ref{central}

\begin{remark}
  The exponent $1/2+1/46$ is best possible using only the conditions
  \eqref{coffee} and \eqref{ineq} that arise from Theorem
  \ref{d(n)K(n)} and Theorem \ref{th-type-3}). Indeed, neither applies
  to the triple $(\mu_{1}, \mu_{2}, \mu_{3})=(13/46, 13/46, 10/23)$.
\end{remark}

\section{Proof of Theorem \ref{onaverage}}\label{6} 

We will now prove Theorem~\ref{onaverage} concerning $d_3$ on integers
congruent to a fixed integer $a\not=0$, modulo $q$, on average over
$q\leq Q$.  We start by elementary reductions.
\par
In addition to the sums $S(\uple{M};q,a)$ and $S(\uple{M})$ which are
defined in \eqref{defS} and \eqref{defSS}, we also consider
$$
S^*(\uple {M},q) = 
\underset{\substack{(m_1m_2m_3 , q)=1}}{\sum\ \sum\ \sum} 
  V_{1}(m_1)V_2(m_2) V_{3}(m_3).
$$
\par
Then, for a prime $q$ satisfying \eqref{q>x1/2} and a triple $\uple
{M}$ satisfying \eqref{MN}, we have
\begin{equation*}
  \frac{1}{q}S(\uple{M}) = \frac{1}{\varphi (q)} S^*(\uple{M};q) + 
  O_{\epsilon }\Bigl( \frac{x^{\frac{99}{100}+\epsilon}}{q} \Bigr) 
\end{equation*}
(compare with \eqref{S=S*}). Using the reductions of \S \ref{prelred}
(in particular Lemma \ref{964}) and Theorem \ref{central}, we see that
Theorem \ref{onaverage} follows from the (equivalent) estimates
\begin{gather}
\label{night0}
\sum_{q\sim Q\atop q\text{ prime},\, q \nmid a} \Bigl\vert S(\uple{M};
q,a) -\frac{1}{q} S(\uple{M})\Bigr\vert  \ll x\LL^{-A},
\\
\label{night1}
\sum_{q\sim Q\atop q\text{ prime},\, q \nmid a} \Bigl\vert S(\uple{M};
q,a) -\frac{1}{\varphi (q)} S^*(\uple{M},q)\Bigr\vert \ll x\LL^{-A}
\end{gather}
are valid for every $A>0$ and $B>0$, every triple $\uple{M}$ (subject
to \eqref{MN}, \eqref{decreasing} and \eqref{eq-mi}), and all $Q$ in a
range
$$
x^{12/23-\alpha}\leq Q\leq x^{9/17-\alpha}
$$
for some $\alpha>0$, where the implied constant may depend only on
$(\alpha,A,B)$ (it would even be enough to do it for each $A$ with $B$
depending on $A$).
\par
We will establish these bounds in two steps: another individual
estimate for each $q$, which follows from the previous sections, and a
final bound on average for which we use Kloostermania~\cite{DI}.

\subsection{Reduction to Kloostermania} 

The first estimate is given by:

\begin{proposition}[Individual bound]\label{1275} 
  With notation as above, for $\uple{M}=(M_1,M_2,M_3)$ satisfying
  \eqref{MN} and \eqref{decreasing}, for every $B>0$, every $\eta>0$
  and $\alpha>0$ and every prime $q$ such that $x^{12/23-\alpha}\leq
  q\leq x^{9/17-\alpha}$ and
$$
q^\frac{5}{2} x^{-1+ \eta} \leq M_{3 }\leq q^{-3} x^{2- \eta}
\quad\text{ or } \quad M_{3}\geq q^\frac{11}{4} x^{-1 + \eta},
$$
we have
\begin{equation}\label{1280}
  S(\uple M; q,a) =
  \frac{1}{q} S(\uple M) + O \bigl( q^{-1} x^{1-\eta_1}),
\end{equation}
for some $\eta_1>0$ depending only on $\eta$, where the implied
constant depends only on $(\eta,\alpha,B)$.
\end{proposition}

\begin{proof} 
This is an immediate consequence of \eqref{coffee} and
\eqref{ineq}.
\end{proof}

Our second estimate is on average over $q$; we will obtain stronger
bounds, and we do not require $q$ to be restricted to primes, but on
the other hand, we now need to fix $a$.

\begin{proposition}[Average bound]\label{proponaver}
  Let $a\not=0$ be a fixed integer.  For every $\eta>0$ there exists
  $\eta_1>0$, depending only on $\eta$, such that for every
  $\uple{M}$ as above satisfying
\begin{equation}\label{11/23>}
x^\frac{11}{23}\geq M_{3}\geq M_{2}\geq M_{1},
\end{equation}
and for every $Q$ such that
\begin{equation}\label{ineqQ}
  Q\leq x^{-\eta} \min \bigl\{xM_3^{-1}, 
  x^{-\frac{1}{2}} M_3^\frac{5}{2}, x^\frac{1}{4} M_3^\frac{3}{4} 
  \bigr\}, 
\end{equation}
we have
\begin{equation}\label{1310}
  \sum_{q\sim Q\atop (q,a)=1} 
  \Bigl\vert S(\uple{M}; q,a) -\frac{1}{\varphi (q)} S^*(\uple{M},q)
  \Bigr\vert \ll x^{1-\eta_1}
\end{equation}
where the implied constant depends only on $(a,B,\eta)$.
\end{proposition}

Before giving the proof, we combine these two results:

\begin{proof}[Proof of~\emph{(\ref{night0})}
  and~\emph{(\ref{night1})}]
Summing \eqref{1280} over all primes $q\sim Q$, we 
obtain \eqref{night0} when
\begin{equation}\label{01/2}
  x^{12/23-\alpha}\leq  Q \leq x^{9/17-\alpha}
\end{equation}
for some fixed $\alpha>0$ and
\begin{equation*}
  Q^\frac{5}{2} x^{-1+ \eta} \leq M_{3}
  \leq Q^{-3} x^{2- \eta} \text{ or } M_{3}\geq Q^\frac{11}{4} 
  x^{-1   + \eta},
\end{equation*}
for some fixed $\eta>0$.
\par
Fixing $\alpha>0$ and $\eta=\alpha$, assuming that~(\ref{01/2}) holds,
it is therefore enough to show that~(\ref{night1}) holds when
\begin{equation}\label{condM3}
Q^{-3} x^{2- \alpha} \leq  M_3 \leq Q^\frac{11}{4} x^{-1 +  \alpha}
\end{equation}
\par
We claim that under these assumptions, if $\alpha$ is small enough,
Proposition~\ref{proponaver} can be applied for the value of the
parameter $\eta=\alpha/2$. We then derive \eqref{1310} by
Proposition~\ref{proponaver} for some $\eta_1>0$, and this is stronger
than~(\ref{night1}).
\par 
To check the claim, note first that the condition~(\ref{11/23>}) is
clear from the assumptions \eqref{01/2} and \eqref{condM3} if $\alpha$
is small enough. Moreover
\begin{itemize}
\item[-] since $M_{3}\leq Q^\frac{11}{4}x^{-1 +\alpha}$ and
  $Q\leq x^\frac{9}{17}$, we have $M_3Q\leq x^{1-1/68+\alpha}\leq 
  x^{1-\eta}$ for $\alpha$ small enough; 
\item[-] since $M_{3}\geq Q^{-3} x^{2-\alpha}$ and $Q\leq
  x^{\frac{9}{17}-\alpha}$, we have $Q\leq
  x^{-\frac{1}{2}-\eta}M_{3}^\frac{5}{2}$, and also $Q \leq
  x^{\frac{1}{4} -\eta}M_{3}^\frac{3}{4}$.
\end{itemize}
\par
This means that~(\ref{ineqQ}) is also valid, as claimed.
\end{proof}

\section{Proof of Proposition \ref{proponaver}}  

We denote by $\Sigma (Q,\uple{M},a)$ on the left-hand side of
\eqref{1310}. Denoting further by $c_{q}$ the sign of the difference
$$
S(\uple{M};q,a) -\frac{S^*(\uple{M})}{\varphi (q)}
$$ 
when $(q,a)=1$, and putting $c_{q}=0$ when $a$ is not coprime to $q$,
we can write
$$
\Sigma (Q,\uple{M},a)=\Sigma_{0} (Q,\uple{M},a) - \Sigma_{1}
(Q,\uple{M}, a),
$$
where
\begin{equation}\label{1341}
  \Sigma_0 (Q,\uple{M},a)= \sum_{q\sim Q} c_{q} S(\uple M;q,a),\quad\quad
  \Sigma_1(Q,\uple{M},a)=\sum_{q\sim Q} \frac{c_{q}}{\varphi (q)}
  S^*(\uple M;q). 
\end{equation}

\subsection{Evaluation  of $\Sigma_{1} (Q,\uple{M},a)$}  
In this section we obtain an asymptotic formula for $\Sigma_1$.

\begin{lemma} 
\label{6.4} 
With notation and assumptions as above, for any complex numbers
$\sigma_q$ with $|\sigma_q|\leq 1$, we have
$$
\sum_{q\sim Q} \frac{\sigma_{q}}{\varphi (q)} S^*(\uple M;q) =
\fourier{V}_1 (0) \fourier{V}_2 (0) \fourier{V}_3 (0) \sum_{q\sim Q}
\frac{\sigma_q}{\varphi (q)}\cdot \Bigl(\frac{\varphi (q)}{q} \Bigr)^3 +O
\bigl(M_2M_3 d^3 (q) \LL^{6B} \bigr),
$$
where $V_i$ are the functions appearing in the definition of
$S(\uple{M};q,a)$.
\end{lemma}

In view of the definition (and the fact that $M_1\leq M_2\leq M_3$),
this follows from the following lemma, which we state in slightly
greater generality for later use:

\begin{lemma}\label{coprimewithq}
  Let $V=V_i$ for some $1\leq i\leq 3$ as in~\emph{(\ref{defVi})}. Then
  for any integer $u\geq 1$ and any integer $q\geq 1$, we have
$$
\sum_{(m_i,q)=1} V_i(um_i) = \frac{ \varphi (q)}{qu} \fourier{V}_i (0)
+ O\bigl( d(q) \LL^{2B}\bigr).
$$
\end{lemma}

\begin{proof}
  If we write $W(t)=V_i(tu)$ for $t\in\Rr$, we see that $W(t)=0$ for
  $\vert t\vert \geq 2M_i/u$ and that $\fourier{W}(t)=(1/u)
  \fourier{V}_i (t/u)$. We then apply the M\"obius inversion formula,
  the Poisson formula (Lemma~\ref{Poisson}) and~\eqref{fourierdecay}
  (with $\nu =2$) to get
  \begin{align*}
    \sum_{(m_i,q)=1} V_i(um_i)&= \sum_{d\mid q\atop d\leq 2M_i/u} \mu
    (d) \sum_{d\mid m_i} W(m_i)= \sum_{d\mid q\atop d\leq 2M_i/u}
    \frac{\mu (d)}{du}  \sum_{ n} \fourier{V}_i\bigl( \frac{n}{du}\bigr)\\
    &= \sum_{d\mid q\atop d\leq 2M_i/u} \frac{\mu (d)}{du} \Bigl\{
    \fourier{V}_i (0) +O\Bigl( M_i\sum_{\vert n \vert \geq 1} \bigl(
    dun^{-1} M_i^{-1} \LL^B \bigl)^2\Bigr)\Bigr\},
\end{align*}
and the lemma follows after summing over $n$ and $d$.
\end{proof}

\subsection{Application of Kloostermania} 

The treatment of $\Sigma_0 (Q, \uple{M},a)$ is more intricate.
Obviously the problem of proving \eqref{1310} essentially deals with
the average distribution of the convolution of two (or three)
arithmetic functions in arithmetic progressions. Thirty years ago,
this problem was considered in a series of papers by Bombieri, Fouvry,
Friedlander and Iwaniec (see in particular \cite{FoIw,Fo84,BFI,BFI2})
with the purpose of improving the exponent $1/2$ in the classical
Bombieri--Vinogradov Theorem concerning the distribution of
primes in arithmetic progressions (see \cite[Theorem 17.1]{IK} for
instance).
\par
These investigations resulted in several variants of the
Bombieri--Vinogradov Theorem, with \emph{well--factorable
  coefficients} in the averaging and with exponents of distribution
\emph{greater} than $1/2$, culminating with the exponent $4/7$
(\cite[Theorem 10]{BFI}). The crucial ingredient was the use of the
so--called \emph{Kloostermania}, i.e., estimates for sums of
Kloosterman sums arising from the Kuznetsov formula and from the
spectral theory of modular forms on congruence subgroups, which was
developed in the seminal work of Deshouillers and Iwaniec \cite{DI}.
\par
Among the currently known results, the following estimate is well
suited to our problem:

\begin{proposition}[Bombieri--Friedlander--Iwaniec]\label{ThmBFI}
  Let $a\not=0$ be an integer.  Let $f$ be a $C^1$ complex-valued
  function defined on $\Rr$ with $|f|\ll 1$. For every $\eta>0$ there
  exists $\eta_1>0$, depending only on $\eta$, such that for every
  sequences $(\gamma_q)$, $(\delta_r)$ and $(\beta_n)$ of complex
  numbers of modulus at most $1$ and for every parameters
$$
x,M,N,Q,R\geq 1
$$
such that $QR<x$, $MN=x$ and
\begin{equation}\label{crazycond}
x^{1-\eta} >M> x^\eta \max\bigl\{
Q, x^{-1} QR^4, Q^\frac{1}{2} R, x^{-2} Q^3 R^4
\bigr\},
\end{equation}
we have
\begin{equation*} 
  \underset{q\sim Q\ r\sim R\atop (qr,a) =1 }
  {\sum\  \sum} \gamma_q \delta_r \Bigl(
  \underset{m\sim M\ n\sim N\atop mn\equiv a \bmod qr}{\sum\  \sum}
  \beta_nf(m) 
  -
  \frac{1}{\varphi (qr)}
  \underset{m\sim M\ n\sim N\atop (mn,  qr)=1}{\sum\  \sum} \beta_nf(m)
  \Bigr) = O\Bigl(x^{1-\eta_1}(1+\sup_{|t|\sim M}|f'(t)|)\Bigr),
\end{equation*}
where the implied constant depends only on $\eta$, $a$ and
$\sup_{t}|f(t)|$. 
\end{proposition}

\begin{proof}
  This follows very easily from~\cite[Theorem 5]{BFI}, which is the
  case $f=1$, after summation by parts; one should just notice that
  the argument in~\cite[p. 235, 236]{BFI} applies equally well when
  $\alpha_m=1$ for $m$ in a sub-interval $I\subset [M,2M]$ and
  $\alpha_m=0$ for $m\sim M$ and $m\notin I$.
\end{proof}

In order to apply this proposition we need to transform
$\Sigma_0(Q,\uple{M},a)$.  For this purpose, we use a trick already
present in \cite[p. 75]{Fo85} (for instance), which consists in
rewriting a congruence to a different modulus: the congruence
$$
m_{1}m_{2}m_{3}\equiv a \bmod q
$$
which appears in our sum $S(\uple{M};q,a)$ (see \eqref{defS}) is
reinterpreted as
\begin{equation}\label{cong}
 qr\equiv -a \bmod m_{1}m_{2}.
\end{equation}
\par
A technical point is that we must preserve the coprimality condition
$(m_{1}m_{2},a)=1$. To avoid complication, we begin with the case
$a=1$, where this technical issue does not arise, and postpone a short
discussion of the general case to Section~\ref{generala}.
\par
For $a=1$, we therefore write
\begin{equation}\label{(6.12)} 
  \Sigma_{0} (Q,\uple{M},1) =
  \underset{m_1 \ \ m_2 }{\sum \ \ \sum}\ V_{1}(m_{1}) V_{2}(m_{2})
  \underset{q\sim Q, \, r\atop qr \equiv -1 \bmod m_{1}m_{2}}{\sum\
    \sum} c_{q}V_{3}\Bigl(\frac{qr +1}{m_{1}m_{2}} \Bigr) .
\end{equation}
\par
By~\eqref{boundderiv} (with $\nu =1$) and~\eqref{MN}, we
have 
$$
V_{3}\Bigl(\frac{qr +1}{m_{1}m_{2}}\Bigr)= V_{3}\Bigl(\frac{qr}
{m_{1}m_{2}}\Bigr) +O\bigl(x^{-1} \LL^{2B}\bigr),
$$ 
and hence
\begin{equation}\label{1360}
  \Sigma_{0} (Q,\uple{M},1) = \underset{m_{1}\ \ m_{2}}{\sum\
    \sum}V_{1}(m_{1}) V_{2}(m_{2}) 
  \underset{q\sim Q, \, r\atop qr \equiv -1 \bmod m_{1}m_{2}}
  {\sum\ \sum} c_{q}V_{3}\Bigl(\frac{qr }{m_{1}m_{2}}
  \Bigr) +O_{B} (\LL^{2B+2}) .
\end{equation} 
\par
This expression is close to the desired shape, but we must separate
the variables $m_1$, $m_2$, $q$ and $r$ before we can apply
Proposition~\ref{ThmBFI}.  We use the Mellin transform for this
purpose.  
\par
First, since $V_3$ is supported in $[M_3, 2M_3]$, the
variable $r$ satisfies
\begin{equation}\label{defR}
  R \ll r \ll R \text{ where } R=M_1 M_2 M_3 Q^{-1}.
\end{equation}
\par
We have
\begin{equation}\label{Mellin2}
  V_{3}(\xi) =\frac{1}{2\pi i}\, 
  \int_{(\sigma)} F_{3}(s)\,  \xi^{-s} ds,
\end{equation}
for any fixed real number $\sigma$, where
$$
F_{3}( s) =\int_{0}^\infty V_{3} (\xi) \xi^{s-1} d \xi
$$
is the Mellin transform of $V_3$. This is an entire function of
$s\in\Cc$ which satisfies
\begin{equation}\label{decayV3}
  F_{3} (\sigma +it) \ll_{k, \sigma} \vert t \vert^{-k}
  M_{3}^{\sigma } \LL^{kB},
\end{equation}
for all $k\geq 1$, all $\sigma\in\Rr$ and $|t|\geq1$ (as follows by
repeated integrations by parts).
\par
Let $\nu>0$ be a small parameter to be chosen later, and let 
\begin{equation*}
T=x^{\nu}
\end{equation*}
\par
Then, inserting \eqref{Mellin2} into \eqref{1360} and applying
\eqref{decayV3} for $k$ large enough depending on $\nu$, we deduce
that
\begin{multline*}
  \Sigma_{0} (Q,\uple{M},1) = \frac{1}{2\pi i} \int_{- i T}^{i T} F_3
  (it)\underset{m_{1}\ \ m_{2}}{\sum \ \ \sum}\bigr( {V_{1}(m_{1})}{
    m_{1}^{it}}\bigl)\cdot \bigl(
  {V_{2}(m_{2})}{m_{2}^{it}}\bigr)\\
  \times \underset{q\sim Q, \, r\atop qr \equiv -1 \bmod
    m_{1}m_{2}}{\sum\ \sum} \bigl( c_{q}q^{-it}\bigr)\cdot r^{-it}
  dt+O(\LL^{2B+2}),
\end{multline*}
where the implied constant depends on $\nu$ and $B$.
\par
For each $t$, we will apply Proposition~\ref{ThmBFI} with
\begin{gather*}
  (Q,R,N,M)\leftrightarrow(M_2,M_1,Q,R),\\
  \gamma_q=V_2(q)q^{it},\quad \delta_r=V_1(r)r^{it},\quad
  \beta_n=c_nn^{-it},\quad m=r,\quad f(m)=F_3(it)m^{-it}.
\end{gather*}
\par
To do this, we must check that the conditions~\eqref{crazycond} are
satisfied for these parameters. For a given $\eta>0$, using~(\ref{defR}),
these conditions translate to
$$
x^{1-\eta}\geq M_1M_2M_3 Q^{-1} \geq x^\eta \max \bigl\{M_2,
x^{-1} M_2 M_1^4, M_2^\frac{1}{2}M_1, x^{-2} M_2^3M_1^4 \bigr\}.
$$
\par
By the assumption \eqref{MN} and the inequality $Q>x^{12/23-\alpha}$,
we see that these inequalities hold as soon as we have
\begin{equation}\label{1461}
  Q \leq x^{-2\eta} \min \bigl\{ xM_2^{-1}, 
  x^2 M_1^{-4} M_2^{-1}, xM_1^{-1}M_2^{-\frac{1}{2}},  x^3M_1^{-4} M_2^{-3}\bigr\}. 
\end{equation}
\par
From $M_1M_2M_3 \leq x$ and $M_1\leq M_2 \leq M_3$, we know that
$M_{1}\leq (x/M_{3})^\frac{1}{2}$), and from this we obtain
\begin{gather*}
  M_2 \leq M_3,\\
  M_1^4 M_2\leq M_1^3 (x/M_3) \leq (x/M_3)^\frac{3}{2} (x/M_3)=
  x^\frac{5}{2} M_3^{-\frac{5}{2}},\\
  M_1M_2^\frac{1}{2} \leq M_1^\frac{1}{2} (x/M_3)^\frac{1}{2} \leq
  (x/M_3)^\frac{1}{4} (x/M_3)^\frac{1}{2}=x^\frac{3}{4} M_3^{-\frac{3}{4}},\\
  M_1^4M_2^3 \leq M_1(x/M_3)^3\leq (x/M_3)^\frac{1}{2} (x/M_3)^3 =
  x^\frac{7}{2} M_3^{-\frac{7}{2}}.
\end{gather*}
\par
Hence~\eqref{1461} is satisfied as soon as we have
$$
Q\ll x^{-2\eta} \min \bigl\{xM_3^{-1}, x^{-\frac{1}{2}}
M_3^\frac{5}{2}, x^\frac{1}{4} M_3^\frac{3}{4}, x^{-\frac{1}{2}
}M_3^\frac{7}{2} \bigr\},
$$
which simplifies into
\begin{equation}\label{Qll}
Q\ll 
x^{-2\eta} \min \bigl\{xM_3^{-1}, 
x^{-\frac{1}{2}} M_3^\frac{5}{2}, x^\frac{1}{4} M_3^\frac{3}{4}   
\bigr\},
\end{equation}
since we have $M_{3}> x^\frac{1}{3}\LL^{-\frac{B}{3}}$.  
\par
This holds by assumption in the setting of
Proposition~\ref{proponaver}, with $\eta$ replaced by $\eta/2$.  After
applying Proposition~\ref{ThmBFI} (noting that $|f(r)|\leq
|F_3(it)|\ll 1$ and $\sup_{r\sim R}{|f'(r)|}\ll T$) we derive
\begin{multline*}
\Sigma_{0}(Q,\uple{M},1)=
\frac{1}{2\pi i} \int_{- i T}^{i T} F_3 (it)\underset{m_{1}\ \
  m_{2}}
{\sum \ \ \sum}\frac{\bigr(  {V_{1}(m_{1})}{ m_{1}^{it}}\bigl)\cdot \bigl(
 {V_{2}(m_{2})}{m_{2}^{it}}\bigr)}{\varphi (m_1m_2)} \\
 \times \underset{q\sim Q, \, r\atop (qr , m_{1}m_{2})=1}{\sum\ \sum} \bigl( c_{q}q^{-it}\bigr)\cdot r^{-it}  dt 
 +O (x^{1-\eta_1+2\nu})
\end{multline*}
where $\eta_1>0$ depends on $\eta$.
\par
Using the Mellin inversion formula again, we then deduce
\begin{equation*}
  \Sigma_0( Q,\uple{M},1)= 
  \underset{m_{1}\ \ m_{2}}{\sum \ \ \sum} \,V_{1}(m_{1}) V_{2}(m_{2})  
  \underset{q\sim Q, \, r\atop (qr , m_1m_2) =1}{\sum\ \sum} 
  \frac{c_{q}}{\varphi (m_1m_2)} V_3 \Bigl( \frac{qr}{m_1m_2}\Bigr)
  + O (x^{1-\eta_1 +2\nu}).
\end{equation*}
\par
Next from Lemma \ref{coprimewithq} we get
$$
\sum_{r\atop (r,m_{1}m_{2})=1} V_{3}\Bigl( \frac{qr}{m_{1}m_{2}}\Bigr)
=\frac{\varphi (m_1m_2)}{q}\cdot \fourier{V}_{3}(0) + O\bigl( d (m_1
m_2) \LL^{2B}\bigr),
$$
and hence finally 
\begin{align}\label{Sigma0final}
  \Sigma_0(Q, \uple{M},1 ) &= \underset{m_1\ m_2}{\sum \ \ \sum}\
  V_{1}(m_{1}) V_{2}(m_{2}) \underset{q\sim Q \atop (q , m_1m_2) =1}
  {\sum} \frac{c_q}{q}\, \fourier{V}_3 (0)+O
  (x^{1-\eta_1+2\nu})+
  O \bigl( Q \LL^{2B+2}\bigr)\nonumber\\
  &= \underset{m_1\ m_2}{\sum \ \ \sum}\ V_{1}(m_{1}) V_{2}(m_{2})
  \underset{q\sim Q \atop (q , m_1m_2) =1} {\sum} \frac{c_q}{q}\,
  \fourier{V}_3 (0)+O (x^{1-\eta_1 +2\nu}),
\end{align}
(if we assume $\eta_1<1/4$, which we can certainly do).
\par
We are now almost done, but before performing the last steps, we will
generalize this formula to an arbitrary integer $a\not=0$. The reader
may skip the next section in a first reading.
  
\subsection{The case of general $a$}\label{generala} 

We will generalize \eqref{Sigma0final} in this section to the sum
$\Sigma_0 (Q,\uple{M},a)$ for a non--zero fixed integer $a$.
\par
For an arbitrary arithmetic function $f(m_{1}, m_{2})$ with bounded
support, we have the decomposition
\begin{equation*}
  \underset{m_1\ \ m_2}{\sum\sum} f(m_{1},m_{2})=
  \sum_{\delta \mid a} \ \sum_{\delta =\delta_1 \delta_2} 
  \sum_{\delta_1 \mid m_1\atop (\frac{m_1}{\delta_1} , \frac{a}{\delta_1})=1}
  \sum_{\delta_2 \mid m_2\atop (\frac{m_2}{\delta_2},\frac{ a}{\delta})=1 } f(m_{1},m_{2})
\end{equation*}
(put $\delta_1=(a,m_1)$, $\delta_2=(a/\delta_1,m_2)$. We apply this
formula to
$$
f(m_{1}, m_{2}) =\underset{q\sim Q, \, r\atop qr \equiv -a \bmod
  m_{1}m_{2}}{\sum\ \sum} c_{q}V_{3}\Bigl(\frac{qr
  +a}{m_{1}m_{2}}\Bigr).
$$
\par
Starting from the analogue of \eqref{(6.12)} for an arbitrary $a$, we
define $a'=a/\delta$, $m'_1= m_1/\delta_1,$ $ m'_2 =m_2/\delta_2$ and
$r'=r/\delta$ and split the congruence \eqref{cong} into $O_a (1)$
sums corresponding to the congruences $qr' \equiv a' \bmod m'_1 m'_2$,
where now we have $(m'_1m'_2, a')=1$ (recall also that $c_q=0$ when
$a$ and $q$ are not coprime). Hence proceeding as before, the formula
\eqref{Sigma0final} generalizes to
\begin{multline}\label{??end}
  \Sigma_0(Q, \uple{M},a ) = \sum_{\delta \mid a}
  \sum_{\delta=\delta_1 \delta_2} \sum_{(m'_1,a/\delta_1) =1}\
  V_{1}(\delta_1m'_{1}) \sum_{(m'_2, a/\delta)=1}V_{2}(\delta_2m'_{2})
  \underset{q\sim Q \atop (q , m'_1m'_2) =1}
  {\sum} \frac{c_q}{q}\, \fourier{V}_3 (0)\\
  +O (x^{1-\eta_1 +2\nu})
\end{multline}
for any fixed integer $a\not=0$.  When $a=1$, this formula becomes
simply~(\ref{Sigma0final}). We thus can continue with it in the
general case.

\subsection{End of the proof}

In~(\ref{??end}), we now exchange the order of the sums, and apply
Lemma \ref{coprimewithq} again to deal with the sums over $m'_1$
(coprime with $aq/\delta_1$) and $m'_2$ (coprime with $aq /\delta$).
By the assumption \eqref{MN} and the bound $M_1\leq M_2\leq M_{3}\leq
x^\frac{11}{23}$, the variables $M_{1}$ and $M_{2}$ are not too small:
we have
\begin{equation}\label{M2M1>}
M_{2}\geq M_{1}\geq \frac{x\mcL^{O(1)}}{M_2M_3}\geq x^\frac{1}{25}.
\end{equation}  
\par
Therefore we have
\begin{multline*}
  \Sigma_0 (Q,\uple{M}, a) = \fourier{V}_1(0) \fourier{V}_2(0)
  \fourier{V}_3(0) \sum_{q\sim Q} \frac{c_q}{q} \sum_{\delta \mid a} \
  \sum_{\delta= \delta_1 \delta_2} \Bigl( \frac{\varphi ((a/\delta_1)
    q)}{(a/\delta_1) q} \cdot \frac{1}{\delta_1}\Bigr) \Bigl(
  \frac{\varphi ((a/\delta ) q)}{(a/\delta ) q}
  \cdot \frac{1}{\delta_2}\Bigr)\\
  + O (x^{1-\eta_1 +2\nu}),
\end{multline*}
provided that (say) $\eta_1 \leq 1/1000$. The sum over $q$ is
restricted to moduli coprime with $a$, and hence writing $a=\delta_1
\delta_2 \delta_3$, we find that the main term of the above expression
is
$$
\fourier{V}_1(0) \fourier{V}_2(0) \fourier{V}_3(0) \sum_{q\sim Q}
\frac{c_q}{q}\cdot \Bigl( \frac {\varphi (q)} {q}\Bigr)^2\
\,\frac{1}{a}\, \underset{a =\delta_1 \delta_2 \delta_3}{\sum\ \sum \
  \sum} \frac{ \varphi (\delta_2 \delta_3) \varphi
  (\delta_3)}{\delta_2\delta_3}.  
$$
\par
Now, an elementary computation gives
$$
\underset{a =\delta_1 \delta_2 \delta_3}{\sum\ \sum \ \sum} \frac{
  \varphi (\delta_2 \delta_3) \varphi (\delta_3)}{\delta_2\delta_3} =
\sum_{d\mid a} \frac{\varphi (d)}{d}\sum_{\delta\mid d} \varphi (d)
=a,
$$
and therefore we get finally
\begin{equation}\label{rty}
  \Sigma_0 (Q,\uple{M}, a) = \fourier{V}_1(0) \fourier{V}_2(0) 
  \fourier{V}_3(0) \sum_{q\sim Q} \frac{c_q}{q}\cdot 
  \Bigl( \frac {\varphi (q)} {q}\Bigr)^2 + 
  O (x^{1-\eta_1 +2\nu}).
\end{equation}
\par
Now gather \eqref{1341}, \eqref{rty} and Lemma \ref{6.4}. The main
terms disappear, and therefore
\begin{equation*} 
  \Sigma (Q,\uple{M}, a) = O (x^{1-\eta_1 +2\nu}),
\end{equation*}
by \eqref{M2M1>}, provided that \eqref{Qll} is satisfied. Now picking
$\nu$ small enough, we obtain Proposition \ref{proponaver}, which
completes the proof of Theorem~\ref{onaverage}.

\end{document}